\documentclass[10pt]{amsart}

\usepackage{amsmath,amssymb}
\usepackage{epsfig,graphicx,latexsym,mathdots}
\usepackage{graphicx,amsfonts,subfigure,dsfont,scalerel}
\usepackage{hyperref}

\theoremstyle{plain}
\newtheorem{theorem}{Theorem}[section]
\newtheorem{lemma}[theorem]{Lemma}
\newtheorem{proposition}[theorem]{Proposition}
\newtheorem{corollary}[theorem]{Corollary}
\newtheorem{definition}[theorem]{Definition}

\theoremstyle{remark}
{
    
    \newtheorem{remark}[theorem]{Remark}

}

\newcommand{\norm}[2]{\left\lVert#1\right\rVert_{#2}}

\newcommand{\CC}{\mathbb{C}}

\newcommand{\lc}{\bar{L}^2(\Omega)}

\renewcommand{\lc}{\bar{H}}

\newcommand{\hc}{\bar{H}^1_0(\Omega)}

\renewcommand{\hc}{\Vb}

\newcommand{\A}{\hat{A}}
\newcommand{\B}{\hat{B}}

\newcommand{\Y}{\mathcal Y}
\newcommand{\Uad}{\mathcal{U}_{ad}}
\newcommand{\Psih}{\hat \Psi}

\newcommand{\xiz}{\xi}

\newcommand{\hatQ}{\widehat \Q}

\def\dd{{\rm d}}
\newcommand{\ddt}{\frac{\rm d}{{\rm d} t} }

\def\weight(#1,#2){c_{#1,#2}}

\def\ra{{\rm ra}}

\if {

}{{\caption}}
} \fi

\def\ph{\hat{p}}

\def\uh{\hat{u}}

\def\xh{\hat{x}}

\def\pb{\bar{p}}

\def\ub{\bar{u}}

\def\xb{\bar{x}}

\def\zb{\bar{z}} 

\def\Hb{\bar{H}}

\def\Ub{\bar{U}}
\def\Vb{\bar{V}}

\def\cala{{\mathcal  A}}
\def\calb{{\mathcal B}}

\def\calh{{\mathcal H}}

\def\calq{{\mathcal Q}}

\def\calu{{\mathcal U}}

\def\calx{{\mathcal X}}
\def\caly{{\mathcal Y}}

\def\cR{{\mathcal R}}

\def\A{\mathcal{A}}
\def\B{\mathcal{B}}

\def\D{\mathcal{D}}

\def\H{\mathcal{H}}

\def\L{\mathcal{L}}
\def\M{\mathcal{M}}

\def\Q{\mathcal{Q}}

\def\T{\mathcal{T}}
\def\U{\mathcal{U}}

\def\X{\mathcal{X}}
\def\Y{\mathcal{Y}}

\newcommand{\Hbc}{\overline{\mathcal H}}

\newcommand{\Xbo}{\overline{X}}
\newcommand{\Ybo}{\overline{ Y}}

\def\eps{\varepsilon}
\def\Om{{\Omega}}

\def\Psih{{\hat\Psi}}
\def\Psih{{\hat\Psi}}

\def\Psih{{\hat\Psi}}

\def\Psih{{\hat\Psi}}

\def\1B{{\bf  1}}

\def\ddiv{\mathop{\rm div}}
\def\dom{\mathop{{\rm dom}}}

\def\Min{\mathop{\rm Min}}

\def\half{\mbox{$\frac{1}{2}$}}

\def\1B{{\bf  1}}

\newcommand{\NN}{\mathbb{N}}
\newcommand{\RR}{\mathbb{R}}

\def\cC{\mathbb{C}}
\def\cN{\mathbb{N}}

\def\cR{\mathbb{R}}

\newcommand\be{\begin{equation}}
\newcommand\ee{\end{equation}}
\newcommand\ba{\begin{array}}
\newcommand\ea{\end{array}}
\newcommand{\bea}{\begin{eqnarray}}
\newcommand{\eea}{\end{eqnarray}}
\newcommand{\bean}{\begin{eqnarray*}}
\newcommand{\eean}{\end{eqnarray*}}

\def\rar{\rightarrow}

\def\disp{\displaystyle}

\def\la{\langle}
\def\ra{\rangle}

\title{Optimal control of PDEs\\ in a complex space setting;\\
application to the Schr\"odinger equation
\thanks{The second and third author were supported by the project "Optimal control of partial differential equations using parameterizing manifolds, model
reduction, and dynamic programming" funded by the Foundation
Hadamard/Gaspard Monge Program for Optimization and Operations
Research (PGMO).}
}

\author{M. Soledad Aronna}
\address{EMAp/FGV, Rio de Janeiro 22250-900, Brazil}
\email{soledad.aronna@fgv.br}

\author{Fr\'ed\'eric Bonnans}
\address{INRIA-Saclay and Centre de 
Math\'ematiques Appliqu\'ees, Ecole Polytechnique, 91128 Palaiseau, France}
\email{Frederic.Bonnans@inria.fr}  

\author{Axel Kr\"oner}
\address{INRIA-Saclay and Centre de 
Math\'ematiques Appliqu\'ees, Ecole Polytechnique, 91128 Palaiseau, France}
\email{Axel.Kroener@inria.fr}  

\if{
\author{M. Soledad Aronna\footnotemark[1]
\and 
Joseph Fr\'ed\'eric Bonnans\footnotemark[2]
\and 
Axel Kr\"oner\footnotemark[3]
}
}\fi 

\begin{document}

\maketitle

\begin{center}
\today
\end{center}

\renewcommand{\thefootnote}{\fnsymbol{footnote}}

\if{
\footnotetext[1]{EMAp/FGV, 
Rio de Janeiro 22250-900, Brazil
 (\href{mailto:soledad.aronna@fgv.br}{mailto:soledad.aronna@fgv.br}).
}

\footnotetext[2]{Inria and CMAP, Ecole Polytechnique, 91128 Palaiseau, France
 (\href{mailto:Frederic.Bonnans@inria.fr}{Frederic.Bonnans@inria.fr}).
}

\footnotetext[3]{Inria and CMAP, Ecole Polytechnique, 91128 Palaiseau, France
 (\href{mailto:Axel.Kroener@inria.fr}{Axel.Kroener@inria.fr}).
}
}\fi

\begin{abstract}
  In this paper we discuss optimality conditions for abstract
optimization problems over complex spaces.
We then apply these results to optimal control problems 
with a semigroup structure. 
As an application we detail the case when the state equation
is the Schr\"{o}dinger one, 
with pointwise constraints on the ``bilinear'' control.
We derive first and second order optimality conditions and address in particular the case that the control enters the state equation and cost function linearly.
\end{abstract}




\pagestyle{myheadings}
\thispagestyle{plain}
\markboth{MARIA SOLEDAD ARONNA, JOSEPH FR\'ED\'ERIC BONNANS, AND AXEL KR\"ONER}{OPTIMAL CONTROL OF PDES IN A COMPLEX SPACE SETTING}


\section{Introduction}
In this paper we derive no gap second order optimality conditions for
optimal control problems in a complex Banach space setting 
with pointwise constraints on the control.
This general framework includes, in particular, optimal control problems for the bilinear Schr\"odinger equation.

Let us consider $T>0$, $\Omega \subset \mathds{R}^n$ an open bounded set,
$n\in  \mathds{N}$, $Q:=(0,T) \times
\Omega$,  and $\Sigma=(0,T)\times \partial \Omega$.
The Schr\"{o}dinger equation is given by
\begin{align}\label{Equ0}
  i\dot \Psi(t,x) + \Delta \Psi(t,x)  - 
u(t) B(x) \Psi(t,x) =0,\quad \Psi(x,0)=\Psi_0(x),
\end{align}
where $t\in (0,T)$, $x\in \Om$, and 
with $u:[0,T] \to \mathds{R}$ the time-dependent electric field,
  $\Psi:[0,T]\times \Omega \to \mathds{C}$ the
wave function, and $B:\Omega \to \mathds{R}$ the 
coefficient of the magnetic field.  
The system describes the
probability of position of a quantum particle 
subject to the electric field $u$; that  will be considered as the control throughout this paper.  The wave function $\Psi$ belongs to the unitary sphere in $L^2(\Om;\CC)$.

For $\alpha_1 \in \RR$ and $\alpha_2 \geq 0$,
the optimal control problem is given as
\begin{equation}
  \label{OC}
  \left\{
  \begin{aligned}
      \min J(u,\Psi):= &\disp 
\half \int_\Omega |\Psi(T)- \Psi_{dT}|^2 dx 
+ \half\int_Q |\Psi
-\Psi_d|^2 dxdt
\\& \disp 
+\int_0^T (\alpha_1 u(t) + \half \alpha_2 u(t)^2 ) dt,
    \text{ subject to 
    \eqref{Equ0} and   $u\in \U_{\text{ad}}$},
  \end{aligned}
  \right.
\end{equation}
with $
\U_{\text{ad}}:=\left\{u \in L^{\infty}(0,T) : u_m \le  u(t) \le u_M \text{ a.e. in } (0,T)\right\}$,
 $u_m,u_M\in \mathds{R},$ $u_m<u_m$ and  
$|z| :=\sqrt{z\bar z}$ for $z\in \mathds{C}$, and 
desired running and final states 
$\Psi_d \colon (0,T) \times \Om \rar \CC$ and 
$\Psi_{dT} \colon \Omega \rar \CC$, resp.
The control of the Schr\"odinger equation is an important question 
in quantum physics.
For the optimal control of semigroups, 
the reader is referred to Li et al. \cite{LiYao:1985,LiYong:1991}, Fattorini et al.  \cite{FattoriniFrankowska:1991,MR2158806} 
and Goldberg and Tr\"oltzsch \cite{MR1227544}.
In the context of optimal control of partial differential equations
for systems in which the control enters linearly
in both the state equation and cost function
(we speak of control-linear problems), 
in a companion paper \cite{aronna:hal-01273496},
we have extended the results of Bonnans \cite{Bonnans:2013} 
(about necessary and sufficient second order optimality conditions
for a bilinear heat equation) to  
problems governed by general bilinear systems in a real Banach space setting,
and presented applications to the heat and wave equation.

The contribution of this paper is the extension to a complex Banach space setting of the optimality conditions of a general class of optimization problems and of the framework developed in \cite{aronna:hal-01273496}. More
precisely, we consider 
optimal control problems governed by a strongly
continuous semigroup 
operator defined in a complex Banach space and derive necessary and
sufficient optimality conditions. 
In particular 
(i) the study of strong solutions when $\alpha_2>0$,
and (ii)
the control-affine case, i.e. when $\alpha_2=0$,
are addressed. The results are applied to the Schr\"odinger equation.

While the literature on optimal control of the heat equation is quite
rich (see, e.g., the monograph by Tr\"oltzsch \cite{MR2583281}),
much less is available for the optimal control of the Schr\"odinger equation. 
We list some references on optimal control of Schr\"{o}dinger equation and related topics. In  Ito and Kunisch \cite{ItoKunisch:2009} necessary optimality conditions are
derived and an algorithm is presented to solve the unconstrained problem, in Baudouin et al.
\cite{BaudouinKavianPuel:2005} 
regularity results for the Schr\"{o}dinger equation with a singular
potential are presented,  further regularity results can be found in Baudouin et al. \cite{BaudouinSalomon:2008} and Boscain et al. \cite{BoscainCaponigroSigalotti:2012} and in particular in Ball et al. \cite{BallMarsdenSlemrod:1982}. For a minimum
time problem and controllability problems for the Schr\"{o}dinger equation see Beauchard et al. \cite{BeauchardCoronTeismann:2014,BeauchardMorancey:2013,Beauchard:2005}. For second order
analysis for control problems of control-affine ordinary differential systems see
\cite{aronna:hal-01081111,Goh66a}.
About the case of optimal control of nonlinear Schr\"{o}dinger equations of Gross-Pitaevskii type arising in the description
of Bose-Einstein condensates, see Hinterm\"uller et
al. \cite{HintermuellerMaharensMarkowichSparber:2013}; for sparse
controls in quantum systems 
see Friesecke et al. \cite{FrieseckeHennekeKunisch:2015}.


The paper is organized as follows. 
In Section \ref{sec:opt_cond_complex} necessary optimality conditions for general minimization problems in complex Banach spaces are formulated.
In Section \ref{sec:abset} the abstract control problem is introduced  in
a semigroup setting and some basic calculus rules are established.
In Section \ref{sec:first_ord} first order optimality conditions, in Section \ref{sec:soocg} sufficient second order optimality conditions are presented; sufficient second order optimality conditions for singular problems are presented in Section \ref{sec:soonoc-goh}, again in a general semigroup setting.
Section \ref{sec:applschr} presents the application, resp. the
control of the Schr\"{o}dinger equation and Section \ref{sec:numerical_example} a numerical tests supporting the possibility
of existence of a singular arc.

\section{Optimality conditions in complex spaces}\label{sec:opt_cond_complex}

\subsection{Real and complex spaces}
We consider complex Banach spaces which  can be identified with the
product of two identical real Banach spaces.
That is, with a real Banach space $X$ we associate the complex
Banach space $\Xbo$ with element represented as 
$x_1 +i x_2$, with $x_1$, $x_2$ in $X$ and 
$i = \sqrt{-1}$, and the usual
computing rules for complex variable, in particular, for
$\gamma=\gamma_1 + i \gamma_2 \in \CC$ with
$\gamma_1$, $\gamma_2$ real, we define 
$\gamma x = \gamma_1 x_1 - \gamma_2 x_2 
+ i (\gamma_2 x_1 + \gamma_1 x_2)$.
Define the real and imaginary parts of a $x \in \Xbo$
by $\Re x$ and $\Im x$, resp.

Let $X$ be a real 
Banach space and $\Xbo$ the corresponding complex one. We denote by 
$\la\cdot,\cdot\ra_X$ (resp. $\la \cdot,\cdot\ra_{\Xbo}$)
the duality product (resp. antiduality product,
which is linear w.r.t. the first argument, and antilinear 
w.r.t. the second).
The dual (resp. antidual) of $X$ (resp. $\Xbo$),
i.e. the set of linear (resp. antilinear) forms, 
is denoted by $X^*$ (resp.~$\Xbo^*$). 

\subsection{Optimality conditions}
We next adress the questions of optimality conditions
analogous to the obtained in the case of real
Banach spaces \cite{MR1756264}.
Consider the problem
\be
\Min_{u,x} f(u,x); \;\; g(u,x) \in  K_g; \;\; 
h(u,x) \in K_h.
\ee
Here $U$ and $W$ are real  Banach space, 
$\Xbo$ and $\Ybo$ are complex Banach spaces,
and $K_g$, $K_h$ are nonempty, 
closed convex subsets of $\Ybo$ and $W$ resp. 
The mappings
$f$, $g$, $h$ from $U\times \Xbo$
to respectively, $\RR$, $\Ybo$, and $W$ are of class $C^1$. 
As said before, the complex space $\Xbo$
can be identified to a pair $X\times X$
of real Banach spaces, with dual 
$X^*\times X^*$. Let
$x^* := (x^*_1,x^*_2) \in X\times X$,
$\xh :=(x^*_1,x^*_2)\in X^*\times X^*$.
Setting 
$x:=x_1+i x_2$ and $x^*:= x^*_1+i x^*_2$
observe that (by linearity/antilinearity of 
$\la\cdot,\cdot\ra_{\Xbo}$)
that 
\be
\la x^*,x\ra_{\Xbo}
=
\la x^*_1,x_1\ra_X+\la x^*_2,x_2\ra_X
+ i \left( 
\la x^*_2,x_1\ra_X-\la x^*_1,x_2\ra_X \right),
\ee
and therefore the `real' duality product in 
$X\times X$ given by
$\la x ^*, x\ra_{X\times X} =
\la x^*_1,x_1\ra_X+\la x^*_2,x_2\ra_X$
satisfies 
\be
\label{xvyfd}
\la x^*,\xh\ra_{X\times X} =
\Re \la x^*,x\ra_{\Xbo}.
\ee
Let $\Xbo$, $\Ybo$ be two complex spaces
associated with the real Banach spaces $X$ and $Y$.
The conjugate transpose of $A\in \L (\Xbo,\Ybo)$
is the operator $A^*\in \L (\Ybo^*,\Xbo^*)$ defined by 
\be
\la y^*,A x \ra_{\Ybo} 
=
\la A^*y^*,x\ra_{\Xbo},
\text{ for all $(x,y^*)$ in $\Xbo\times \Ybo^*$.}
\ee
If $A\in L(U,\Ybo)$, identifying the real Banach space $U$
with the space of real parts of the corresponding 
complex Banach space $\Ub$, we may define 
$A^*\in \L (\Ybo^*,\Ub^*)$ by
\be
\la A^*y^*,u\ra_{\Ub} 
=
\la y^*,Au\ra_{\Ybo}.
\ee
Combining this relation with 
\eqref{xvyfd}, we deduce that 
\be
\Re\la y^*,Au\ra_{\Ybo} 
=
\Re\la A^*y^*,u\ra_{\Ub} 
=
\la\Re A^* y^*,u\ra_U.
\ee
We deduce the following expression of normal cones,
for $y\in \Ybo$:
\be
N_{K_g}(y) =\{ y^*\in \Ybo^*; \;\; 
\Re \la y^*,z-y\ra_{\Ybo} \leq 0,
\text{ for all $z\in K_g$}\}.
\ee
For  $\lambda\in \Ybo$ and $\mu\in W$
the Lagrangian of the problem is defined as
\be
L(u,x,\lambda,\mu)
 := f(u,x) +\Re \la\lambda,  g(u,x) \ra_{\Ybo} 
+ \la\mu, h(u,x) \ra_W.
\ee

\begin{lemma}
\label{aocl}
The partial derivatives of the Lagrangian 
are as follows:
\be
\left\{ \ba{lll} \disp
\frac{\partial L}{\partial u} & = \disp
 \frac{\partial f}{\partial u}
+\Re  \left(\frac{\partial g}{\partial u}^*  \lambda \right)
+  \frac{\partial h}{\partial u}^\top \mu,
\\ \disp
\frac{\partial L}{\partial x_r} & = \disp
\frac{\partial f}{\partial x_r}
+\Re \left( \frac{\partial g}{\partial x}^*  \lambda\right)
+  \frac{\partial h}{\partial x_r}^\top \mu,
\\ \disp
\frac{\partial L}{\partial x_i} & = \disp
\frac{\partial f}{\partial x_i}
+ \Im\left(\frac{\partial g}{\partial x}^*  \lambda\right)
+  \frac{\partial h}{\partial x_i}^\top \mu.
\ea\right.\ee
In particular, we have that
\be
\disp
\frac{\partial L}{\partial x_r} 
+ i
\frac{\partial L}{\partial x_i} 
=
\frac{\partial f}{\partial x_r} 
+ i
\frac{\partial f}{\partial x_i} 
+
\frac{\partial g}{\partial x}^*  \lambda
+
\left( \frac{\partial h}{\partial x_r}
+ i
\frac{\partial h}{\partial x_i}\right)^\top \mu.
\ee
\end{lemma}

\begin{proof}
We have that, skipping arguments:
\be
\ba{lll} \disp
\frac
{\partial L}{\partial u}v
&= \disp
\frac
{\partial f}{\partial u}v
+\Re(\lambda,  \frac{\partial g}{\partial u}v  )_{\Ybo} 
+ ( \mu, \frac{\partial h}{\partial u}v  )_W
\vspace{1mm}
\\&= \disp
\frac {\partial f}{\partial u}v
+\Re(\frac{\partial g}{\partial u}^*  \lambda,  v  )_U
+ ( \frac{\partial h}{\partial u}^\top \mu, v  )_U
\vspace{1mm}
\\&= \disp
\left( \frac{\partial f}{\partial u}
+\Re ( \frac{\partial g}{\partial u}^*  \lambda )
+  \frac{\partial h}{\partial u}^\top \mu, v  \right)_U
\ea\ee
for all $v \in U$.
We have used that setting 
$\frac{\partial g}{\partial u}= a +ib$
and
$\lambda=\lambda_r + i \lambda_i$, 
 then 
\be
\ba{lll} \disp 
\left( \Re \left(\frac{\partial g}{\partial u}^* 
 \lambda \right) , v\right)_U
&=
(\Re ( a^\top-i b^\top) (\lambda_r + i \lambda_i) , v)_U
\\ &=
(a^\top \lambda_r + b^\top  \lambda_i), v)_U
\\ & =
\Re\left(\frac{\partial g}{\partial u}^*  \lambda,  v  \right)_U.
\ea\ee
Now, for $z_r \in X$: 
\be
\ba{lll} \disp
\frac
{\partial L}{\partial x_r}z_r
&= \disp
\frac
{\partial f}{\partial x_r}z_r
+\Re(\lambda,  \frac{\partial g}{\partial x}z_r  )_{\Ybo} 
+ ( \mu, \frac{\partial h}{\partial x_r}z_r  )_W
\\&= \disp
\frac
{\partial f}{\partial x_r}z_r
+\Re(\frac{\partial g}{\partial x}^*  \lambda,  z_r  ) _{\Xbo}
+ ( \frac{\partial h}{\partial x_r}^\top \mu, z_r  )_{\Xbo}
\\&= \disp
\left(\frac{\partial f}{\partial x_r}
+\Re ( \frac{\partial g}{\partial x}^*  \lambda )
+  \frac{\partial h}{\partial x_r}^\top \mu, z_r  \right)_{\Xbo}
\ea\ee
and for all $z_i \in X$:
\be
\ba{lll} \disp
\frac
{\partial L}{\partial x_i}z_i
&= \disp
\frac
{\partial f}{\partial x_i}z_i
+\Re(\lambda,  \frac{\partial g}{\partial x_i}i z_i  )_{\Ybo} 
+ ( \mu, \frac{\partial h}{\partial x_i}z_i  )_W
\\&= \disp
\frac
{\partial f}{\partial x_i}z_i
-
\Re(i\frac{\partial g}{\partial x_i}^*  \lambda,  z_i  ) _{\Xbo}
+ ( \frac{\partial h}{\partial x_i}^\top \mu, z_i  )_{\Xbo}
\\&= \disp
\frac
{\partial f}{\partial x_i}z_i
+
\Im(\frac{\partial g}{\partial x}^*  \lambda,  z_i  ) _{\Xbo}
+ ( \frac{\partial h}{\partial x_i}^\top \mu, z_i  )_{\Xbo}
\\&= \disp
\left(\frac{\partial f}{\partial x_i}
+ \Im(\frac{\partial g}{\partial x}^*  \lambda)
+  \frac{\partial h}{\partial x_i}^\top \mu, z_i  \right)_{\Xbo}.
\ea\ee
The result follows.
\end{proof}

\begin{remark}
Not surprisingly, we obtain the same optimality system as if we had 
represented the constraint $g(u,x)=0$ 
as an element of the product of real spaces. 
The advantage of the complex setting is to allow
more compact formulas.
\end{remark}

\section{The abstract control problem in
a semigroup setting}\label{sec:abset}

Given a complex Banach space $\Hbc$,
 we consider optimal control problems for equations of type
\be
\label{semg1}
\dot\Psi + \cala \Psi = f + u (\B_1 + \B_2\Psi); 
\;\; t\in (0,T); \quad \Psi(0)=\Psi_0,
\ee
where  
\be
\label{semg1-hyp}
\Psi_0\in \Hbc;\;\;  f\in L^1(0,T;\Hbc);\;\; 
\B_1\in \Hbc;\;\; 
u\in L^1(0,T);\;\; 
\B_2\in \L(\Hbc),
\ee
and $\A$ is the generator of a strongly continuous semigroup on 
$\Hbc$, 
in the sense that, denoting by  $e^{-t\A}$ the semigroup 
generated by $\A$, we have that 
\be
\label{def-domA}
\dom(\A) := \left\{ y\in \Hbc; \;\; \lim_{t\downarrow 0} 
\frac{y-e^{-t\A}y}{t} \;\text{exists} \right\}
\ee
is dense 
and for $y\in \dom(\A)$, $\A y$ is equal to the above
limit. Then $\A$ is closed.
Note that we choose to define $\A$ and not 
its opposite as the generator of the semigroup.
We have then 
\be
\label{cond:bounded}
\| e^{-t \A}\|_{\L(\Hbc)} \leq c_\A e^{\lambda_\A t},\quad t>0,
\ee
for some positive $c_\A$ and $\lambda_\A$.
For the semigroup theory in a complex space setting we
refer to Dunford and Schwartz 
\cite[Ch. VIII]{DunSchI58}.
The solution of 
\eqref{semg1} in the {\em semigroup sense}
is the function
$\Psi\in C(0,T;\Hbc)$ such that,
for all $t\in [0,T]$:
\be
\label{semPsi}
\Psi(t) = e^{-t\A}\Psi_0 + 
\int_0^t e^{-(t-s)\A} \big( f(s)+ u(s) (\B_1+ \B_2 \Psi(s)) \big) \dd s.
\ee 
This fixed-point equation \eqref{semPsi} is well-posed
in the sense that it has a unique solution in 
$C(0,T;\Hbc)$, see \cite{aronna:hal-01273496}.
We recall that the conjugate transpose
of $\A$ has domain
\be
\dom(\A^*) := \{ \varphi\in \Hbc^*; \;\; \text{for some $c>0$:}\;\;
| \la \varphi, \A y \ra | \leq c \|y\|, \;\; \text{for all $y\in \dom(\A)$}\},
\ee
with antiduality product $\la \cdot ,\cdot \ra:=\la \cdot ,\cdot \ra_{\Hbc}$. Thus, $y\mapsto \la \varphi, \A y \ra$
has a unique extension to a linear continuous form over $\Hbc$,
which by the definition is $\A^*\varphi$. 
This allows to define weak solutions, extending to the complex
setting the definition in  \cite{Ball:1977}:

\begin{definition}
\label{ThmAC-def}
We say that $\Psi\in C(0,T;\Hbc)$ is a
{\em  weak solution} of \eqref{semg1}
if  $\Psi(0)= \Psi_0$ and, for any 
$\phi \in \dom(\A^*)$,  the function $t\mapsto \la \phi , \Psi(t)\ra$ 
is absolutely continuous over $[0,T]$ and satisfies
\begin{align}
\label{weak_ball}
\ddt\la \phi , \Psi(t)\ra+\la \A^*\phi , \Psi(t)\ra 
= \la \phi , f+u(t)(\B_1+\B_2 \Psi(t))\ra,
\;\; \text{for a.a. $t\in [0,T]$.}
\end{align}
\end{definition}

We recall the following result,  obvious extension to the 
complex setting of the corresponding result in \cite{Ball:1977}:

\begin{theorem}
\label{weaksense-thm}
Let $\A$ be the generator of a strongly continuous semigroup.
Then there is a unique weak solution of \eqref{weak_ball} that coincides with the
semigroup solution.
\end{theorem}

So in the sequel we can use any of the two equivalent formulations 
\eqref{semPsi} or \eqref{weak_ball}.
 The control and state spaces are, respectively,
\be
\calu:=L^1(0,T); \quad \caly:=C(0,T;\Hbc).
\ee
For $s\in [1,\infty]$ we set 
$\calu_s := L^s(0,T)$.
Let $\uh \in \calu$ be given and $\Psih$ solution of \eqref{semg1}.
The
{\em  linearized state equation}
 at $(\Psih,\uh)$,  to be understood in the 
semigroup sense, is
\be
\label{semg5lin}
\dot z(t) + \A z(t) = \uh(t) \B_2 z(t) + v(t) (\B_1+\B_2\Psih(t));
\quad z(0)=0,
\ee
where $v\in \U.$ It is easily checked that 
given $v\in \U,$ the equation \eqref{semg5lin} 
has a unique solution denoted by $z[v]$,
and that the mapping 
$u\mapsto \Psi[u]$ from $\U$ to $\Y$ is of class $C^\infty$,
with 
$D\Psi[u]v= z[v]$.

The results above may allow to prove higher regularity.
\begin{definition}[Restriction property]
\label{def:restr_prop}
Let $E$ be a Banach space, with norm denoted by
$\|\cdot\|_{E}$ with continuous inclusion in $\Hbc$. 
Assume that the restriction of $e^{-t\A}$ to $E$ 
has image in $E$, and that it is 
a continuous semigroup over this space.
We let $\A'$ denote its associated generator,
and $e^{-t\A'}$ the associated semigroup.
By \eqref{def-domA} 
we have that 
\be
\label{def-domAw}
\dom(\A') := \left\{ y\in E; \;\; \lim_{t\downarrow 0} 
\frac{e^{-t\A}y-y}{t} \;\text{exists} \right\}
\ee
so that $\dom (\A') \subset \dom (\A)$,
and $\A'$ is the restriction of $\A$ to 
$\dom(\A')$. We have that 
\be
\label{cond:bounded2}
\| e^{-t \A'}\|_{\L(E)} \leq c_{\A'} e^{\lambda_{\A'} t}
\ee
for some constants 
$c_{\A'}$ and $\lambda_{\A'}$.
Assume that $\B_1 \in E$, 
and denote by $\B'_2$ the restriction of $\B_2$ 
to  $E$, which is supposed to have image in 
$E$ and to be continuous in the topology of $E$,
that is,
\be
\label{lem-reg.l1-1}
\B_1 \in E; \quad \B'_2 \in \L(E). 
\ee
In this case we say that $E$ has the 
{\em restriction property}. 
\end{definition}

\subsection{Dual semigroup}\label{sec:absetdual}
Since $\Hbc$ is a reflexive Banach space it is known, e.g. 
\cite[Ch. 1, Cor. 10.6]{MR710486},  
that
$\A^*$ generates another strongly continuous
semigroup called the {\em dual (backward) semigroup} on $\Hbc^*$,
denoted by $e^{-t \A^*}$, which satisfies 
\be
\label{etatransp}
(e^{-t \A})^* = e^{-t \A^*}. 
\ee
The reference \cite{MR710486} above assumes a real setting, but the arguments
have an immediate extension to the complex one.
Let $(z,p)$ be solution of the
forward-backward system 
\be
\label{yp}
\left\{\ba{rcl}
{\rm (i)} & 
\dot z + \A z &= az + b,
\\ {\rm (ii)} & 
-\dot p + \A^*p &=a^* p + g,
\ea\right.\ee
where 
\be
\label{yp-bis}
\left\{\ba{l}
b\in L^1(0,T;\Hbc), \\
g\in L^1(0,T;\Hbc^*),\\
a \in L^\infty(0,T;\L(\Hbc)),
\ea\right.\ee
and for a.a. $t\in (0,T)$, $a^*(t)$ is the conjugate transpose  operator of $a(t)$,
element of $L^\infty(0,T;\L(\Hbc^*))$.

The solutions of \eqref{yp} in the semigroup sense are 
$z\in C(0,T;\Hbc)$, $p\in C(0,T;\Hbc^*)$, 
and for a.a. $t\in (0,T)$: 
\begin{equation}
\label{pqform6}
\left\{ \begin{aligned}
&{\rm (i)}\,\, & z(t) &= e^{-t \A} z(0) +  \int_0^t e^{-(t-s) \A} (a(s)z(s)+b(s)) \dd s,
\\
&{\rm (ii)} \,\,&  p(t)& = e^{-(T-t) \A^*} p(T) +  \int_t^T e^{-(s-t) \A^*} (a^*(s)p(s)+g(s)) \dd s.
\end{aligned}
\right.
\end{equation}
The following integration by parts (IBP) lemma follows:

\begin{lemma}\label{lem:IBP1}
 \label{pqform}
Let $(z,p) \in C(0,T;\Hbc) \times C(0,T;\Hbc^*) $ satisfy
\eqref{yp}-\eqref{yp-bis}. Then,
\be
\label{pqform0}
\la p(T), z(T)\ra+ \int_0^T \la g(t), z(t)\ra \dd t 
=
\la p(0), z(0)\ra + 
\int_0^T \la p(t), b(t)\ra \dd t.
\ee
\end{lemma}

\begin{proof}
This is an obvious extension of
 \cite[Lemma 2.9]{aronna:hal-01273496}
to the complex setting.
\end{proof}

\section{First order optimality conditions}\label{sec:first_ord}

\subsection{The optimal control problem}\label{sec:notation}

Let $q$ and $q_T$ be continuous quadratic forms over $\Hbc$,
with associated symmetric and continuous 
operators 
 $Q$ and $Q_T$ in $\L(\Hbc,\Hbc^*)$, such that 
$q(y)=\Re \la Qy,y\ra$ and $q_T(y)=\Re \la Q_Ty,y\ra$,
where the operators $Q$ and $Q_T$ are self-adjoint, i.e., 
\be
\la Q x,y \ra = \overline{\la Q y,x \ra}
\quad
\text{for all $x$, $y$ in $\Hbc$. }
\ee
Observe that the derivative of $q$ at $y$ in direction $x$ is
\be
Dq(y)x = 2 \Re \la Qy,x\ra.
\ee
Similar relations for $q_T$ hold. 

\begin{remark}
The bilinear form associated with the quadratic form $q$ is
\be
\half( q(x+y)-q(x)-q(y)) = \Re \la Qx,y\ra.
\ee
Then 
\be
\Im \la Qx,y\ra = \Re ( -i \la Qx,y\ra) = \Re \la Qx,iy\ra
=
\half( q(x+iy)-q(x)-q(iy)).
\ee
\if{
It follows that
\be
\la Qx,y\ra = \half( q(x+y)-q(x)-q(y)) + \half i ( q(x+iy)-q(x)-q(iy)).
\ee
So, 
\be
\la Qx,iy\ra = \half( q(x+iy)-q(x)-q(iy)) + \half i ( q(x-y)-q(x)-q(-y)).
\ee
} \fi 
\end{remark}

Given
\be
\label{psi_d-psi-dt}
\Psi_d\in L^\infty(0,T;\Hbc); \quad \Psi_{dT}\in\Hbc,
\ee
we introduce the cost function, where
$\alpha_1 \in \RR$ and $\alpha_2 \geq 0$,
assuming that 
$u\in L^2(0,T)$ if $\alpha_2\neq 0$:
\be
\label{def-cost-fun}
J(u, \Psi) := 
\int_0^T  ( \alpha_1  u(t) + \half \alpha_2  u(t)^2)  \dd t 
+
\half\int_0^T q( \Psi(t)-\Psi_d(t) )\dd t 
+
 \half q_T(\Psi(T)-\Psi_{dT})
\ee
The costate equation is
\be
\label{semg8}
- \dot p + \A^* p = Q(\Psi-\Psi_d) + u \B_2^* p; 
\quad
p(T) = Q_T(\Psi(T)-\Psi_{dT} ).
\ee
We take the solution in the (backward)
semigroup sense:
\be
\label{semadj}
p(t) = e^{ (t-T) \A^*} Q_T(\Psi(T)-\Psi_d(T) )
+ \int_t^T e^{ (t-s) \A^*} 
\big( Q(\Psi(s)-\Psi_d(s) ) + u(s) \B_2^* p(s) \big) \dd s.
\ee
The reduced cost is
\be
F(u) := J(u,\Psi[u]). 
\ee
The set of feasible controls is 
\be\label{Uad}
\calu_{ad}:= \{ u\in\calu; \; u_m \leq u(t) \leq u_M \; \text{a.e. on } [0,T] \},
\ee
with $u_m<u_M$ given real  constants. 
The optimal control problem is 
\be\label{problem:OC}
\tag{P}
\Min_u F(u); \quad u\in \calu_{ad}.
\ee
Given 
$(f,y_0)\in L^1(0,T;\Hbc)\times \Hbc$, 
let $y[y_0,f]$ denote
the solution in the semigroup sense of 
\be
\dot y(t) + \A y(t)= f (t),\quad
t\in (0,T), 
\qquad y(0)=y_0. 
\ee
The {\em compactness hypothesis} is
\be
\label{compachyp}
\left\{ \ba{lll}
\text{For given $y_0\in \calh$,
the mapping $f\mapsto \B_2 y [y_0,f]$}
\\
\text{is compact from
$L^2(0,T;\Hbc)$ to $L^2(0,T;\Hbc)$.}
\ea\right.\ee

\begin{theorem}
\label{existsol}
Let \eqref{compachyp} hold.  
Then problem \eqref{problem:OC} has a nonempty set of solutions. 
\end{theorem}

\begin{proof}
Similar to \cite[Th. 2.15]{aronna:hal-01273496}.
\end{proof}

We set
\be
\label{expr-lambda}
\Lambda(t):=\alpha_1+\alpha_2 \uh(t)  + \Re \la p(t),\B_1+\B_2\Psih(t) \ra.
\ee

\begin{theorem}
\label{diffcostfct}
The mapping $u\mapsto F(u)$ is of class 
$C^\infty$ from $\U$ to $\cR$ and we have that 
\be
\label{diffcostfct1}
DF(u)v=\int_0^T \Lambda(t)  v(t) \dd t,\qquad \text{for all } v\in \U.
\ee
\end{theorem}

\begin{proof}
That $F(u)$ and $J$ are of class $C^\infty$
follows from classical arguments based on the implicit
function theorem, as in \cite{aronna:hal-01273496}.
This also implies
that, setting $\Psi:=\Psi[u]$ and $z:=z[u]$: 
\be\ba{lll}
DF(u)v = & \disp
\int_0^T (\alpha_1 + \alpha_2 u(t) ) v(t) )\dd t
+
\int_0^T \Re \la Q( \Psi(t)-\Psi_d(t)), z(t) \ra \dd t 
\vspace{1mm} \\ & +
\Re \la Q_T( \Psi(T)-\Psi_{dT}), z(T) \ra.
\ea\ee
We deduce then \eqref{diffcostfct1} from lemma \ref{pqform}.
\end{proof}
Let for $u \in \Uad$ and $I_{m}(u)$ and $I_{M}(u)$ be the associated contact sets defined, up to a zero-measure set, as
\be
\left\{ 
\begin{aligned}
  I_{m}(u) &:= \{t\in (0,T) : u(t) = u_m\},
\\
   I_{M}(u) &:= \{t\in (0,T) : u(t) = u_M \}.
\end{aligned}
\right.
\ee
The first order optimality necessary condition is given as follows.
\begin{proposition}
\label{Prop1order}
 Let $\uh$ be a local solution of problem \eqref{problem:OC}.
Then, up to a set of measure zero there holds
 \be\label{first_ord_nec}
 \{t;\; \Lambda(t)>0\}\subset I_m(\uh),\quad \{t;\; \Lambda(t)<0\}\subset I_M(\uh).
 \ee
\end{proposition}
\begin{proof}
 Same proof as in  \cite[Proposition 2.17]{aronna:hal-01273496}.
\end{proof}

\section{Second order optimality conditions}\label{sec:soocg}
\subsection{Technical results}

Set $\delta \Psi := \Psi-\hat\Psi.$ 
Since $u \Psi-\uh\Psih= u \delta\Psi + v \Psih$,
we have, in the semigroup sense:
\be
  \label{semdeltaPsi}
  \ddt \delta \Psi(t)+ \A \delta \Psi(t) = 
\uh(s) \B_2 \delta\Psi(s) + v(t) (\B_1 + \B_2 \Psih(t) + \B_2 \delta\Psi(s)).
\ee
Thus, $\eta := \delta\Psi-z$ is solution of 
\be
  \label{sem-eta}
  \dot \eta(t) + \A  \eta(t) = \uh \calb_2 \eta(t) + v(s)\B_2 \delta\Psi(s).
\ee

We get the following estimates.
\begin{lemma}
\label{Lemmaestz}
The linearized state $z$ solution of \eqref{semg5lin},
the solution $\delta \Psi$ of \eqref{semdeltaPsi}, and 
$\eta = \delta\Psi-z$ solution of \eqref{sem-eta} satisfy,
whenever $v$ remains in a bounded set of $L^1(0,T)$:
\begin{eqnarray}
\label{estz}
\|z\|_{L^\infty(0,T;\Hbc)} &=& O( \|v\|_1),
\\
\label{estdeltaPsi}
\|\delta\Psi\|_{L^\infty(0,T;\Hbc)} &=& O( \|v\|_1),
\\
\label{esteta}
\|\eta\|_{L^\infty(0,T;\Hbc)} &=& O(\| \delta\Psi\,v\|_{L^1(0,T;\Hbc)}) = O(\|v\|_1^2).
\end{eqnarray}
\end{lemma}

\begin{proof}
Similar to the proof of lemma 2.18 in 
\cite{aronna:hal-01273496}. 
\end{proof}

For $(\Psih,\uh)$ solution of \eqref{semg1},  
$\ph$ the corresponding solution of \eqref{semadj}, 
$v\in L^1(0,T)$, and $z\in C(0,T;\Hbc)$, 
let us set
\be
  \label{tildeQ}
  \Q(z,v) := \int_0^T \Big( q (z(t) ) + \alpha_2 v(t)^2 
 + 2  v(t) \Re \la \ph(t),\B_2  z(t) \ra  \Big)\dd t + q_T (z(T) ).
  \ee

\begin{proposition}\label{ExpEasy}
  Let $u$ belong to $\calu$. 
Set $v:=u-\uh$, 
  $\Psih := \Psi [\uh]$, $\Psi := \Psi [u]$.
Then
  \be
  F(u) = F(\uh) + D F(\uh)v + \half  \Q(\delta\Psi,v).
  \ee
  
\end{proposition}

\begin{proof}
 We can expand the cost function as follows:
\be
\label{semcost}
  \begin{split}    
F(u)= & \disp 
F(\uh) 
+ \half \int_0^T \alpha_2 v(t)^2 +
q(\delta\Psi(t)) )\dd t
+ \half q_T(\delta\Psi(T) ) 
\\ & \disp 
+ \int_0^T  (\alpha_1+\alpha_2 \uh(t)) v(t) \dd t 
\\ & \disp 
+ \Re \left( \int_0^T 
\la Q( \hat \Psi(t) -\Psi_d(t)), \delta \Psi) \ra \dd t  
+ \la Q_T( \hat\Psi(T) -\Psi_d(T)),\delta \Psi(T) \ra \right).
  \end{split}
\ee
Applying lemma \ref{lem:IBP1}
to the pair $(\delta\Psi,\ph)$, where $z$ is solution of the 
linearized equation \eqref{semg5lin},
and using the expression of $\Lambda$ in 
\eqref{expr-lambda}, we obtain the result.
\end{proof}

\begin{corollary}\label{cor:expansion2}
We have that
\be
F(u) = F(\uh) + D F(\uh)v + \half  \Q(z,v) + O(\|v\|^3_1),
\ee 
where $z := z[v].$
\end{corollary}

\begin{proof}
We have that
\be
\ba{lll}
\Q(\delta\Psi,v) -\Q(z,v) &=&  \disp
\Re\left( \int_0^T
\la Q(\delta \Psi(t) +z(t)),\eta(t)\ra 
+ 2 v(t) \la p(t), B_2 \eta(t) \ra \dd t \right)\\ &&+
\Re\left(
\la Q_T (\delta \Psi(T)+z(T)),\eta(T)\ra \right).
\ea
\ee
By \eqref{estz}-\eqref{esteta} 
this is of order of $\|v\|^3_1$.
The conclusion follows.
\end{proof}
 
\subsection{Second order necessary optimality conditions}

Given a feasible control $u$, the critical cone is defined as
\be
C(u) :=
\left\{
\begin{split}
  & v\in L^1(0,T) \,|\ \Lambda(t) v(t)=0\text{ a.e. on } [0,T], \\
  & v(t)\ge 0\,\, \text{a.e. on } I_m( u),\; v(t)\le 0\text{ a.e. on }I_M( u)
\end{split}\right\}.
\ee

\begin{theorem} 
\label{thm:nec_sec_ord_cond}
Let $\uh\in \calu$ be a local solution
of \eqref{problem:OC} 
and $\ph$ be the corresponding costate. Then there holds,
  \be
  \Q(z[v],v) \ge 0\quad \text{for all } v \in C(\uh).
  \ee
\end{theorem}

\begin{proof}
The proof is similar to the one of theorem
3.3 in \cite{aronna:hal-01273496}. 
\end{proof}

\subsection{Second order sufficient optimality conditions}
In this subsection we assume that 
$\alpha_2 >0$, and obtain
second order sufficient optimality conditions.
Consider the following condition: there exists
$\alpha_0 >0$ such that
\be
\label{thm:nec_sec_ord_cond-suff}
  \Q(z,v) \geq \alpha_0 \int_0^T v(t)^2 \dd t,
\quad \text{for all } v \in C(\uh).
  \ee

\begin{theorem} 
\label{thm:nec_sec_ord_cond-2}
Let $\uh\in \calu$ satisfy the first order optimality conditions
of \eqref{problem:OC},
$\ph$ being the corresponding costate,
as well as \eqref{thm:nec_sec_ord_cond-suff}
Then $\uh$ is a local solution of problem $(P)$,
that satisfies the quadratic growth condition. 
\end{theorem}

\begin{proof}
It suffices to adapt the arguments in say 
\cite[Thm. 4.3]{MR2683898} or
Casas and Tr{\"o}ltzsch \cite{MR3470757}. 
\end{proof}

Using the technique of 
Bonnans and Osmolovski{\u\i} \cite{MR2731283}
we can actually deduce from  theorem
\ref{thm:nec_sec_ord_cond} 
that $\uh$ is a strong solution 
in the following sense (natural extension of the 
notion of strong 
solution in the sense of the calculus of variations).

\begin{definition}
We say that a control $\uh\in \calu_{ad}$
is a strong solution if there exists $\eps>0$
such that, if $u\in \calu_{ad}$ and
$\| y[u]-y[\uh] \|_{C(0,T;\Hbc)} < \eps$, then 
$F(\uh) \leq F(u)$. 
\end{definition}

In the context of optimal control of PDEs, sufficient conditions
for strong optimality were recently
obtained for elliptic state equations in 
Bayen et al. \cite{MR3152723},
and for parabolic equations by 
Bayen and Silva \cite{MR3478951}, and by 
Casas and Tr{\"o}ltzsch \cite{MR3470757}.

We consider the part of the Hamiltonian  depending on the control: 
\be
H(t,u) := \alpha_1 u +  \half \alpha_2 u^2 + u \Re \la \ph(t), \B(t) \ra,
\ee
where $\B(t):=\B(t)_1+\B(t)_2 \hat \Psi(t)$.
The Hamiltonian inequality reads
\be
H(t,\uh(t)) \leq H(t,u), \quad \text{for all $u\in [u_m,u_M]$,
for a.a. $t\in [0,T]$.}
\ee
Since $\alpha_2>0$, $H(t,\cdot)$ is a strongly convex function,
and therefore the Hamiltonian inequality follows from the
first order optimality conditions and in addition we have the
quadratic growth property
\be
\label{hutquadgr}
H(t,\uh(t)) + \half \alpha_2 (u-\uh(t))^2 \leq H(t,u), 
\quad \text{for all $u\in [u_m,u_M]$,
for a.a. $t\in [0,T]$.}
\ee

\begin{lemma}
\label{cvhamuk}
Let $\uh$ be feasible and satisfy the 
first order optimality conditions,
with $\alpha_2>0$.
Let $u_k$ be also feasible 
such that the associated states $\Psih_k := \Psi [u_k]$
converge to $\Psih$ in $C(0,T;\Hbc)$, and
$\limsup_k F(u_k) \leq F(\uh)$. Then 
$u_k\rar \uh$ in $L^2(0,T)$. 
\end{lemma}

\begin{proof}
Since $u_k$ is bounded in $L^\infty(0,T)$, 
from the expression of the cost function of
the optimal control problem in view of theorem \ref{diffcostfct} and corollary \ref{cor:expansion2}, it follows that
\be
0 \geq \limsup_k (F(u_k)-F(\uh)) 
=\limsup_k \int_0^T ( H(t,u_k(t))- H(t,\uh(t)) ) \dd t.
\ee
Then the conclusion follows from 
the quadratic growth property 
\eqref{hutquadgr}.
\end{proof}

For $u_k$ as in Lemma \ref{cvhamuk} we have
\be
B_k := \{t\in (0,T); \; |u_k(t)-\uh(t)| > \sqrt{ \|u_k-\uh\|_1} \};
\quad A_k := (0,T)\setminus B_k.
\ee
Note that 
\be
\label{est-bk}
| B_k | \leq \int_0^T \frac{|u_k(t)-\uh(t)| }
{ \sqrt{\|u_k-\uh\|_1} } \dd t
=
\sqrt{\|u_k-\uh\|_1 }.
\ee
Set for a.a. $t$:
\be
v^A_k(t) := (u_k(t)-\uh(t)) \1B_{A_k}(t); \quad
v^B_k(t) := (u_k(t)-\uh(t)) \1B_{B_k}(t).
\ee
We now extend to the semigroup setting the 
{\em decomposition principle} from 
\cite{MR2731283},
which has been extended to the elliptic setting 
by 
\cite{MR3152723},
and to the parabolic setting by
\cite{MR3478951}.

\begin{theorem}[Decomposition principle]
\label{thm:dp}
For $u_k$ as in Lemma \ref{cvhamuk} we have that
$| B_k| \rar 0$, and 
\be
F(u_k) = F(\uh+ v^A_k) + F(\uh+ v^B_k) - F(\uh) +o(\|u_k-\ub\|^2_2).
\ee
and also 
\be
F(\uh+v^B_k) - F(\uh) =
\int_{B^k} ( H(t,u_k(t))- H(t,\uh(t)) ) \dd t
+o(\|u_k-\ub\|^2_2).
\ee
\end{theorem}

\begin{proof}
Remember the linearized state equation \eqref{semg5lin}
whose solution is denoted by $z[v]$.
Set 
\be
v_k:=u_k-\uh; \;\; z_k:=z[v_k]; \;\; 
z^A_k:=z[v^A_k]; \;\; z^B_k:=z[v^B_k].
\ee
Since $A_k\cap B_k$ has null measure, we have that
$z_k=z^A_k + z^B_k$. 
Also, 
\be
\label{estimvB}
\|v^B_k\|_1 \leq 
 | B_k|^{1/2} \|v^B_k\|_2 = o(\|v^B_k\|_2),
\ee
since $|B_k| \rar 0$ by lemma \ref{cvhamuk}.
Then, in view of lemma \ref{Lemmaestz}:
\be
\label{estimzB}
\|z^B_k\|_{C(0,T;\Hbc)} = O( \|v^B_k\|_1) = o(\|v^B_k\|_2).
\ee
Combining with corollary \ref{cor:expansion2}
and using the fact that $v^A_k(t) v^B_k(t)=0$
a.e., we deduce that
\be
\ba{lll}
F(u_k) -F(\uh) 
&= &
DF(\uh)v_k + \half \calq(v_k,z_k) + o(\|v_k\|^2_2)
\\ & = & \disp 
DF(\uh)v_k + \half \calq(v_k,z^A_k) + o(\|v_k\|^2_2)
\\ & = & \disp 
DF(\uh)v^A_k + \half \calq(v^A_k,z^A_k) 
+DF(\uh)v^B_k + \half \alpha \|v^B_k\|^2_2 
\\ &  & \disp
+2 \int_0^T v_k^B(t) \Re \la \ph(t), \B_2 z^A_k(t) \ra \dd t
+ o(\|v_k\|^2_2)
\\ & = & \disp 
 DF(\uh)v^A_k + \half \calq(v^A_k,z^A_k) 
+ DF(\uh)v^B_k + \half \alpha \|v^B_k\|^2_2 
+ o(\|v_k\|^2_2),
\ea\ee
where we have used the fact that, by \eqref{estimvB}:
\be
\left| \int_0^T v_k^B(t) \Re \la \ph(t), \B_2 z^A_k(t) \ra \dd t \right|
=
O(
\| v^B_k \|_1 \| z_k^A\|_{C(0,T;\Hbc)} ) = o(\| v_k \|^2_2).
\ee

Now
\be
F(\uh+ v^A_k) -F(\uh) = DF(\uh)v^A_k + \half \calq(v^A_k,z^A_k) 
+ o(\|v^A_k\|^2_2),
\ee
and by \eqref{estimzB}
\be
F(\uh+ v^B_k) -F(\uh) = DF(\uh)v^B_k  + \half \alpha_2 \|v^B_k\|^2_2
+ o(\|v^B_k\|^2_2).
\ee

Combining the above relations we get the desired result.
\end{proof}

\begin{definition}
We say that $\uh$ satisfies the
\emph{quadratic growth condition for strong solutions}
if there exists $\eps>0$ and $\eps'>0$ such that 
for any feasible control $u$: 
\be
F(\uh) + \eps \|u-\uh\|^2_2 \leq F(u),
\quad 
\text{whenever $\|\Psi[u]-\Psi[\ub]\|_{C(0,T;\Hbc)} < \eps'.$}
\ee
\end{definition}

\begin{theorem}
Let $\uh$ satisfy the first order necessary 
optimality condition \eqref{first_ord_nec}, and the second order sufficient condition
\eqref{thm:nec_sec_ord_cond-suff}. 
Then $\uh$ is a strong minimum that satisfies the above
quadratic growth condition.
\end{theorem}

\begin{proof}
If the conclusion is false, then there exists a sequence $u_k$
of feasible controls such that 
$\Psi_k\rar \Psih$ in $C(0,T;\Hbc)$, where $\Psi_k:=\Psi[u_k]$, and 
$F(u_k) \leq F(\uh) + o(\|u_k-\uh\|_2^2).$
By lemma \ref{cvhamuk},
$u_k\rar \uh$ in $L^2(0,T)$. 
By the decomposition theorem \ref{thm:dp} 
and since $DF(\uh) v^B_k\geq 0$, 
it follows that 
\be
\label{expansfab}
 \alpha_2 \|v^B_k\|^2_2 + 
F(\uh+v^A_k ) - F(\uh) \leq  o(\|v_k\|^2_2). 
\ee
We next distinguish two cases.
\\ (a) 
Assume that 
$\|v^A_k\|_2/ \|v_k\|_2 \rar 0$. 
We know that
\be
F(\uh+v^A_k ) - F(\uh) =
DF(\uh) v^A_k + \half \calq(v^A_k,z^A_k) +o(\|v^A_k\|^2_2).
\ee
Since (by the first order optimality conditions)
$DF(\uh) v^A_k \geq 0$ and 
$\calq(v^A_k,z^A_k) =  O(\|v^A_k\|^2_2 )=o(\|v_k\|^2_2)$ by hypothesis,
it follows with \eqref{expansfab} that
$\|v^B_k\|^2_2 = o(\|v_k\|^2_2) =  o(\|v^B_k\|^2_2)$
which gives a contradiction.
\\ (b) Otherwise, 
$\liminf_k \|v^A_k\|_2/ \|v_k\|_2 > 0$
(extracting if necessary a subsequence).
It follows from \eqref{expansfab} that 
\be
\label{expansfabb}
F(\uh+v^A_k ) - F(\uh) \leq  o(\|v^A_k\|_2). 
\ee
Since $\|v^A_k\|_\infty \rar 0$,
we obtain a contradiction with theorem
\ref{thm:nec_sec_ord_cond}.
\end{proof}

\begin{remark}
A shorter proof for theorem \ref{thm:dp} is obtained by 
combining lemma \ref{cvhamuk} and the Taylor expansion in 
corollary \ref{cor:expansion2}, which implies
\be
F(u) = F(\uh) + D F(\uh)v + \half  \Q(z,v) + O(\|v\|^3_2),
\ee 
from which we can state a sufficient condition for optimality in 
$L^2(0,T)$. 
On the other hand the present proof opens the way for dealing with
non quadratic (w.r.t. the control) Hamiltonian functions, 
as in \cite{MR3152723}.
\end{remark}

\section{Second order optimality conditions for singular problems}\label{sec:soonoc-goh}

In this section we assume that
$\alpha_2=0$, so that the control enters linearly in both the 
state equation and cost function. 
For such optimal control problems there is an extensive theory in the 
finite dimensional setting, see 
Kelley \cite{Kel64}, Goh \cite{Goh66a}, Dmitruk \cite{Dmi77,MR914861}, 
Poggiolini and Stefani \cite{MR2472886}, Aronna et al.
\cite{ABDL12}, and Frankowska and Tonon
\cite{frankowska:hal-01067270}; the case of additional scalar state
constraints was considered in Aronna et
al. \cite{aronna:hal-01081111}. 

In the context of optimal control of PDEs, there exist very few
papers on sufficient optimality conditions for affine-linear control
problems, see Bergounioux and Tiba \cite{MR1377719}, Tr\"oltzsch
\cite{MR2144184}, Bonnans and Tiba \cite{MR2656166},
Casas \cite{MR2974742}
(and the related literature involving $L^1$ norms, see e.g.
Casas et al. \cite {MR3032866}).
As mentioned in the introduction, here we will follow the ideas in 
\cite{aronna:hal-01273496,Bonnans:2013} 
by using in an essential way the Goh transform \cite{Goh66a}.

Let $E_1\subset \H$ with continuous inclusion, 
having the restriction property
(definition~\ref{def:restr_prop}).
We can denote the restriction 
of $\B_2$ to $E_1$ by $\B_2$ with no risk of confusion. 
In the rest of the paper we make the following hypothesis:
\be
\label{hyp-Goh-tr1a}
\left\{ \ba{lll}
{\rm (i)} &
\B_1 \in \dom (\A),\;  
 \\ {\rm (ii)} &
\B_2 \dom(\A) \subset \dom(\A),\quad 
\B_2^* \dom(\A^*) \subset \dom(\A^*),\;\; 
\ea\right.
\ee
with $\B^k_i := (\B_i)^k$.
So, we may define the operators
below, with domains $\dom(\A)$ and $\dom(\A^*),$ respectively:
\be
\left\{
\begin{split}
[\A,\B^k_2] &:=\A \B^k_2-\B^k_2\A ,\\ 
[(\B^k_2)^*,\A^*]&:=(\B^k_2)^*\A^* - \A^*(\B^k_2)^*.
\end{split}
\right.\ee

\be
\label{hyp-Goh-tr1b}
\left\{ \ba{lll}
{\rm (i)} &
\text{For $k=1,2$, 
$\left[\A, \B^k_2\right]$ has a continuous
extension to $E_1$,}
\\
& \text{denoted by $M_k$, 
}
\\ {\rm (ii)} &
f \in L^\infty(0,T;\Hbc);  \quad 
M^*_k \ph\in L^\infty(0,T;\Hbc^*), \; k=1,2,
\\ {\rm (iii)} &
\Psih\in L^2(0,T;E_1); \;\; 
[M_1,\B_2] \Psih \in L^\infty(0,T;\Hbc).
\ea\right.\ee

\begin{remark}
Point \eqref{hyp-Goh-tr1a} (ii) implies
\be
\B^k_2 \dom(\A) \subset \dom(\A),\quad 
(\B^k_2)^* \dom(\A^*) \subset \dom(\A^*),\qquad \text{for } k=1,2.
\ee
So, $[\A, \B_2]$ 
is well-defined as operator
with domain $\dom(\A)$, and point (iii) makes sense.
\end{remark}

We also assume that 

\be
\label{hyp-Goh-tr2}
\left\{ 
\ba{lll}
{\rm (i)} & 
\B^2_2 f\in C(0,T;\Hbc);
\; \; \Psi_d \in C(0,T;\Hbc), 
\\[1.5ex] {\rm (ii)} & 
M^*_k \ph \in C(0,T;\Hbc^*), \;\; k=1,2.
\ea\right.
\ee 
Let $\xi\in C(0,T;\Hbc)$ be (semigroup) solution of the
following equation 
\be
\label{diffeqxi}
\dot \xi + \A \xi = \uh \B_2 \xi + w b^1_z;
\quad \xi(0) =0,
\ee
where 
\be
\label{equ-bunz}
b^1_{z} 
:= -\B_2 f -  M_1 \Psih  - \A \B_1.
\ee
Note that $b^1_{z} \in C(0,T;\Hbc)$, so that equation
\eqref{diffeqxi} has a unique solution.  
Consider the space
\be
\label{def--ww}
W := \left(L^2(0,T;E_1) \cap C([0,T];\H) \right)
\times L^2(0,T) \times \cR.
\ee
We define the continuous quadratic forms 
over $W,$ defined by 
\be
\label{Omega}
\hatQ(\xi,w,h) = \hatQ_T(\xi,h) +\hatQ_a(\xi,w)+\hatQ_b(w),
\ee
where 
$\hatQ_b(w):= \int_0^T w^2(t) R(t)\dd t$ and 
\begin{align}\label{OmegaT}
\hatQ_T(\xi,h)&:= q_T(\xi(T) + h \calb(T) )
 + h^2  \Re \la \ph(T),\B_2\B_{1}+\B_2^2 \Psih(T)\ra +h 
\Re \la \ph(T),\B_2 \xi(T) \ra,
\\
\hatQ_a(\xi,w)&:=  \Re \int_0^T \Big( q(\xi) + 2 w  \la Q\xi,\B \ra
 + 2 w\la  Q(\Psih-\Psi_d),\B_2\xiz\ra - 
2w \la M^*_1 \ph, \xiz\ra   
\Big) \dd t,
\end{align}
with $R \in L^\infty(0,T)$ given by
\be
\label{R}
\left\{ 
\begin{split}
R(t)&:=  q(\B) + \Re \la Q(\Psih-\Psi_d),\B_2\B \ra +  \Re \la\ph(t),r(t) \ra,
\\
r(t) &:= \B_2^2 f(t)  -\A\B_2\B_1+2\B_2\A\B_1 
- 
\big[M_1,\B_2\big] \Psih.
\end{split}
\right. 
\ee

We write $PC_2(\uh)$ for the closure in the $L^2\times \RR$--topology  of the set
\be
PC(\uh) := \{(w,h) \in W^{1,\infty}(0,T) \times \RR, \dot w \in 
C(\uh); \; w(0)=0, \; w(T)=h  \}.
\ee
The final value of $w$ becomes an independent variable when we
consider this closure.

\begin{definition}[Singular arc]
  The control $\uh(\cdot)$ is said to have a {\em singular arc} in a nonempty interval $(t_1,t_2) \subset [0,T]$  if, 
for all $\theta>0,$ there exists $\eps>0$ such that
  \be
  \uh(t) \in
  [u_m+\varepsilon,u_M-\varepsilon],\quad \text{for a.a. } t\in (t_1+\theta,t_2-\theta).
  \ee
  We may also say that $(t_1,t_2)$ is a singular arc itself.
  We call $(t_1,t_2)$ a {\em lower boundary arc} if $\uh(t)=u_m$ for
  a.a. $t \in (t_1,t_2)$, and an {\em upper boundary arc} if $\uh(t)=u_M$
  for a.a. $t\in (t_1,t_2)$. 
We sometimes simply call them boundary arcs.
We say that a boundary arc $(c,d)$ is {\em initial} if $c=0$, and {\em final} if $d=T$. 
\end{definition}

\begin{lemma}\label{lem3}
For $v\in L^1(0,T)$ and $w \in AC(0,T)$,
$w(t) =\int_0^t v(s) \dd s$, 
 there holds 
  \be
  \Q(z[v],v)=\hatQ(\xi[w],w,w(T)).
  \ee
For any $(w,h) \in L^2(0,T)\times \RR$: 
  \be
\label{lem3-ome}
  \hatQ(\xi[w],w,h) \ge 0\quad \text{for all } (w,h) \in PC_2(\uh).
  \ee
In addition, provided the mapping
\begin{align}\label{comp-hyp-xi}
w\mapsto \xi[w],\quad 
L^2(0,T) \rar L^2(0,T;\Hbc)
\end{align}
is compact we have that
$R(t)\geq 0$ a.e. on singular arcs.
\end{lemma}
\begin{proof}
 Similar to \cite[Lemma 3.9 and corollary 3.11]{aronna:hal-01273496}.
\end{proof}

In the following we assume that the following hypotheses hold: 
\begin{enumerate}
\item {\em finite structure:}
\be\label{finite_structures}
\left\{
\begin{array}{l}
\text{there are finitely many boundary and singular maximal arcs} \\ \text{and the closure of their union is $[0,T],$}
\end{array}
\right.
\ee
\item {\em strict complementarity} for the control constraint
(note that $\Lambda$ is a continuous function of time)
\be
\label{strict_complementarity}
\left\{ \begin{array}{l}
\text{$\Lambda$ has nonzero values over the interior of each boundary
  arc, and}
\\
\text{at time 0 (resp. $T$) if an initial (resp. final) boundary
  arc exists,}
\end{array}\right.
\ee
\end{enumerate}

\begin{proposition}
Let \eqref{finite_structures}--\eqref{strict_complementarity} hold.
Then
\be
PC_2(\uh)=
\left\{
\begin{array}{l}
(w,h)\in L^2(0,T)\times \cR;\text{ $w$ is constant over boundary arcs,}\\ \text{ $w=0$ over an initial boundary arc }\\
\text{and $w= h$ over a terminal boundary arc}
\end{array}
\right\}.
\ee
\end{proposition}

\begin{proof}  
Similar to the one of \cite[Lemma 8.1]{ABDL12}. 
\end{proof}

Letting $\T_{BB}$ denote the {\em set of bang-bang junctions,} we 
assume in addition that 
\be
\label{RposTBB}
R(t) >0,\quad t\in \T_{BB}.
\ee

Consider the following positivity condition:
there exists $\alpha>0$ such that
\be
\label{sufcondso}
\Om(\xi[w],w,h) \geq \alpha ( \|w\|^2_2 + h^2),
\quad\text{for all $(w,h) \in PC_2(\uh)$.}
\ee
We say that $\uh$ satisfies a {\em weak quadratic growth condition}
if there exists $\beta>0$ such that for any  $u \in \U_{ad},$
setting $v:=u-\uh$ and $w(t):=\int_0^Tv(s)\dd s,$ we have
\be
\label{sufcondqg}
F(u) \geq F(\uh) + \beta ( \|w\|^2_2 + w(T)^2),
\quad\text{if $\|v\|_1$ is small enough.}
\ee
The word `weak' makes reference to the fact that the growth is
obtained for the $L^2$ norm of $w$, and not the one of $v$. 

\begin{theorem}
Let \eqref{finite_structures}-\eqref{strict_complementarity}
and \eqref{RposTBB} 
hold. Then \eqref{sufcondso} holds iff  the quadratic growth condition
\eqref{sufcondqg}  is satisfied.
\end{theorem}

\begin{proof}
Similar to the one in 
\cite[Thm 4.5]{aronna:hal-01273496}. 
\end{proof} 

\vspace{20mm}
\section{Application to Schr\"odinger equation}\label{sec:applschr}

\subsection{Statement of the problem} The equation is formulated first in an informal way.
Let $\Omega\subset \RR^n$, $n\in \NN$, open and bounded, and $T>0$.
The state equation, with $\Psi=\Psi(t,x)$, is
\be
\left\{
\begin{aligned}
  \label{Schroedinger_equ}
  \dot\Psi(t,x) - i 
\sum_{j,k=1}^n \frac{\partial}{\partial x_k} 
\left[a_{jk}(x) \frac{\partial \Psi(t,x)}{\partial x_j}  \right]
& =- i u b_2 \Psi(t,x) + f 
&&\text{in } (0,T) \times \Om,\\
    \Psi(0,x)&=\Psi_0&& \text{in } \Omega, \\
    \Psi(t,x) &=0  && \text{on } (0,T) \times \partial \Om
  \end{aligned}
  \right.
\ee
with 
\begin{equation}
\begin{aligned}\label{assump_Schroed}
  &\Psi_0\in \Vb, \;\;
  b_2^k \in W^{2,\infty}_0(\Omega),\; k=1,2,
  &f \in L^2(0,T; \Vb) \cap C(0,T;\Hb)\\ 
\end{aligned}
\end{equation}
 and the complex-valued spaces $\Hb := L^2(\Omega;\cC)$ 
and $\Vb:= H^1_0(\Omega;\cC)$.
Note that although $f$ is normally equal to zero, it is
useful to introduce it since the sensitivity of the solution
w.r.t. the r.h.s., that plays a role in the numerical analysis.
Here the $a_{jk}$ are $C^1$ functions over 
$\bar\Om$ that satisfy, for each $x\in \bar\Om,$
the symmetry hypothesis 
$a_{jk} = a_{kj}$ for all $j$, $k$
as well as the following coercivity hypothesis, that for some $\nu>0$:
\be
\sum_{j,k=1}^n a_{jk}(x) \xi_j \xi_k \ge \nu |\xi|^2,\quad 
\text{for all $\xi\in \CC^n$, $x\in\Om$.}
\ee
 We apply the abstract setting
with $\Hbc=\Hb$.
Consider the unbouded operator in $\Hb$
defined by 
\be
\label{A_elliptic-r}
(\A_0 \Psi)(t,x):= -
\sum_{j,k=1}^n \frac{\partial}{\partial x_k} 
\left[a_{jk}(x) \frac{\partial \Psi(t,x)}{\partial x_j}  \right],
\quad (t,x)\in (0,T)\times \Om,
\ee
with domain 
$\dom( \A_0):= \Hb^2(\Om)\cap \Vb$, where $\Hb^2(\Om)$ denotes the complex valued Sobolev space $H^2(\Om,\CC)$.
One easily checks that this operator is self-adjoint,
i.e., equal to the conjugate transpose.
\if{ defined by 
\be
\la \A^*_0\Psi, \Phi\ra_{\Hb} = 
\la \Psi , \A_0 \Phi\ra_{\Hb},
\quad \text{for all $\Psi\in \dom(\A_0)$}.
\ee
with domain 
\be
\dom(A^*_0) := \{ \Phi \in \Hb; \;\;
\text{for some $c>0$: }
| (\A_0 \Psi, \Phi)_{\Hb} | \leq c \| \Psi \| \text{ for all }\Psi \in \dom(\A_0)\}.
\ee
} \fi 
The PDE \eqref{Schroedinger_equ} enters in the 
semigroup framework, with generator 
\be
\label{A_elliptic}
(\A_{\Hb} \Psi):= i \A_0 \Psi,
\quad \text{for all $\Psi \in \Hb$.}
\ee

\begin{lemma}
The operator $\A_{\Hbc}$, with domain 
$\dom( \A_{\Hbc} ):= \Hb^2(\Om)\cap \Vb$, 
is the generator of a unitary semigroup and
\eqref{Schroedinger_equ} 
has a semigroup solution $\Psi \in C(0,T;\Hb)$.
\end{lemma}
\begin{proof}
That  $\A_{\Hbc}$ is the generator of a contracting semigroup
follows from 
the Hille Yosida characterization 
with $M=1$, $n=1$ and $\omega=0$.
The operator $\A_{\Hbc}$ being the opposite of its conjugate transpose 
it follows that the semigroup is norm preserving.
\end{proof}
We define then the following sesquilinear form over $\Vb$:
\be
a(y,z) := \sum_{j,k=1}^n \int_\Om 
a_{jk}(x)  \frac{\partial y }{\partial x_j } 
 \frac{\partial \zb}{\partial x_k } \dd x,
\quad \text{for all $y$, $z$ in $\Vb$},
\ee
which is self-adjoint in the sense that
\be
\overline{a(y,z)} = a(z,y).
\ee
Furthermore, for $y$, $z$ in $\dom(\A_0)$
we have that 
\be
\label{a0yz}
\la  \A_0 y,z \ra_{\Hb} = a(y,z) 
= 
\overline{a(z,y)} = \la  y, \A_0 z \ra_{\Hb}.
\ee
so that is $\A_0$ also self-adjoint.

 \subsection{Link to variational setting and regularity for Schr\"odinger equation}
 
We introduce the function space
 \be
  \label{spaceXa}
\calx:=
L^\infty(0,T;\Vb)\cap H^1(0,T; \Vb'),
  \ee
endowed with the natural norm
\be
\|\Psi\|_{\calx} :=
\| \Psi\|_{L^\infty(0,T;\Vb)}
+
\| \Psi\|_{H^1(0,T; \Vb')}.
\ee

There holds the following weak convergence result.
\begin{lemma}
\label{lemwcv}
Let $\Psi_k$ be a bounded sequence in 
$\calx$. Then there exists $\Psi\in \calx$ 
such that a subsequence of $\Psi_k$ 
converges  to $\Psi$ strongly in $L^2(0,T;\lc)$,
and weakly in $L^2(0,T;\Vb)$,
and $H^1(0,T; \Vb')$.
Finally, if $u_k$ weakly$*$ converges to 
$u$ in $L^\infty(0,T)$, then 
\be
\label{lemwcv1} 
u_k b_2 \Psi_k \rar u b_2 \Psi
\quad \text{weakly in  $L^2(0,T;\bar H)$}
\ee
\end{lemma}

\begin{proof}
By the Aubin-Lions lemma \cite{MR0152860},
$\calx$ is compactly embedded into 
$L^2(0,T;\Hb)$. Thus, 
extracting a subsequence if necessary, we may assume 
that $\Psi_k$ converges in $L^2(0,T;\lc)$ to some 
$\Psi$. Since $\Psi_k$ is bounded in the 
Hilbert spaces $L^2(0,T;\Vb)$
and $H^1(0,T; \Vb')$, re-extracting a subsequence 
if necessary, we may assume that it also
weakly converges in these spaces. 

Let $C_R$ denote the closed ball of 
$L^\infty(0,T,\Vb)$ of radius $R$. 
This is a closed subset of 
$L^2(0,T,\Vb)$ that, for large enough $R$,
contains the sequence $\Psi_k$. 
Since any closed convex set is weakly closed, 
$\Psi \in C_R$. 
 Thus $\Psi \in \X.$
That \eqref{lemwcv1} holds follows from the 
joint convergence of 
$u_k$ in $L^\infty(0,T)$
(endowed with the weak$*$ topology),
and of $\Psi_k$ in $L^2(0,T;\bar H)$.
\end{proof}

The variational solution of \eqref{Schroedinger_equ} is given as
$\Psi\in \calx$ satisfying, for a.a. $t\in (0,T)$:
\be
\label{eq-vf}
\la \dot \Psi(t) , z \ra_{\Vb} + i a( \Psi(t), z) 
+ iu(t) \la b_2\Psi,z\ra_{\Hb} = \la f(t), z \ra_{\Vb}
 \text{ for all $z\in \Vb$},
\ee
and $\Psi(0)=\Psi_0 \in \Vb$.

For $(f,b_2,u,\Psi_0)\in L^2(0,T;\Vb)\times W^{1,\infty}(\Omega)\times L^{\infty}(\Om) \times \Vb $ we
set 
\be
\ba{lll}
  \kappa[f,b_2,u,\Psi_0]= &
\norm{f}{L^1(0,T; \Vb) }^2 +
  \norm{\Psi_0}{\Vb }^2 + 
\\ & 
\norm{u}{L^\infty(0,T)}^2\norm{\nabla b_2}{L^\infty(\Om)}^2 (\norm{f}{L^2(0,T;\Hb)}^2 + \norm{\Psi_0}{\Vb}^2) .
\ea  \ee
There holds the following existence and regularity result for the
unique solution of
\eqref{eq-vf} 
 (cf.~\cite{LionsMagenes:1972}).

\begin{theorem}
\label{thm:reg}
  Let $(f,b_2,u,\Psi_0)\in L^2(0,T;\Vb)\times W^{1,\infty}(\Omega)\times L^{\infty}(\Om) \times \Vb $. 
  Then there exists 
$c_0>0$ independent of $(f, b_2, u, \Psi_0)$ 
such that 
\eqref{eq-vf} 
has a unique solution $\Psi$ in $\calx$, that satisfies the estimates
  \begin{gather}
    \label{L2_estimate}
    \|\Psi\|_{C(0,T;\lc)} \leq
    c_0 \big(\norm{f}{L^1(0,T;\Hb)} + \norm{\Psi_0}{\lc}
    \big),
\\
    \label{str_estimate}
    \|\Psi\|_{C(0,T;\Vb)} +
    \|\dot\Psi(t)\|_{L^2(0,T;\Vb')} \leq c_0 \kappa[f,b_2,u,\Psi_0].
  \end{gather}
\end{theorem}

\begin{proof}
Since $\Om$ is bounded, there exists a
Hilbert basis of $H^1_0(\Om)$ 
$(w_j,\lambda_j)$, $j \in \cN$ of (real) 
eigenvalues and 
nonnegative eigenvectors of the operator 
$\A_0$
(with, by the definition, homogeneous Dirichlet conditions), i.e
\be
  \label{basewj}
  - \sum_{j,k=1}^n \frac{\partial}{\partial x_k} 
\left[a_{jk}(x) \frac{\partial w_j(x)}{\partial x_j}  \right]
 = \lambda_j w_j(x),\quad j=1,\cdots,
w_j \in H^1_0(\Om), \quad \lambda_j \in \RR_+.
  \ee
Consider the associated Faedo-Galerkin discretization method;
that is, let 
$\{ \Vb_k \}$ be the finite dimensional subspaces of $\Vb$
generated by the (complex combinations of the)
$w_j$, for $j\leq k$.  
The corresponding approximate solution 
$\Psi_k(t)=\sum_{j=1}^k \psi^j_k(t) w_j$ 
of \eqref{Schroedinger_equ},
with $\psi^j_k(t) \in \CC,$ 
 is defined as the solution of
  \be
  \label{Psiktest1-a}
\la \dot \Psi_k(t),  w_j \ra_{\Hb} + i a(\Psi_k(t),  w_j) 
+iu(t)\la b_2\Psi_k(t), w_j \ra_{\Hb} =
\la f(t), w_j \ra_{\Hb},
\ee
for $j=1,\dots,k$ and $t\in [0,T]$,
with initial condition
\be\label{ode0}
\psi_k^j(0)= (\Psi_0,w_j),\quad \text{for } j=1,\dots,k.
\ee
For each $k\in \cN,$ the above equations  are a
system of linear ordinary differential equations that has a unique
solution 
$\psi_k=(\psi_k^1,\dots,\psi_k^k)\in C(0,T; \CC^k)$. 
It follows that for any 
$\Phi(t) = \sum_{j=1}^k \phi^j(t) w_j$ 
(where $\phi^j(t) \in L^1(0,T)$ for $j=1,\dots,k$)
we have that 
  \be
  \label{Psiktest1}
\la \dot \Psi_k(t),  \Phi(t) \ra_{\Hb} +  i a(\Psi_k(t),  \Phi(t)) 
+iu(t) \la b_2\Psi_k(t), \Phi(t) \ra_{\Hb} = 
\la f(t), \Phi(t)\ra_{\lc},
\ee
We derive {\em a priori} estimates by using different test
functions $\Phi$. 

\begin{enumerate}
  \item Testing with $\Phi(t)= \Psi_k(t)$ gives 
    \be
    \label{reg2}
    \la \dot\Psi_k(t),\Psi_k(t) \ra_{\Hb} + i a(\Psi_k(t), \Psi_k(t)) 
+ iu(t) \la b_2\Psi_k(t), \Psi_k(t) \ra_{\Hb} =
    \la f(t),\Psi_k(t) \ra_{\lc}.
    \ee
Taking the real part in both sides in \eqref{reg2} we obtain
    \be
    \half \frac{\dd}{\dd t} \|\Psi_k(t)\|_{\lc}^2  \leq C_1
    \|f(t)\|_{\lc}\|\Psi_k(t)\|_{\lc} \le C_2(\|f(t)\|_{\lc}^2+\|\Psi_k(t)\|_{\\\lc}^2). 
    \ee
    By Gronwall's inequality we get the following estimate
    \be
\label{L2_estimatek}
    \|\Psi_k\|^2_{L^\infty(0,T;\lc)} \leq 
    C_3(\norm{f}{L^1( 0,T;\lc)}^2 + \|\Psi_k(0)\|^2_{\lc}).
    \ee

  \item 
Testing with 
$\Phi(t)= \sum_{j=1}^k \lambda_j \psi^j_k(t) w_j = \A_0 \Psi_k(t)$
gives 
    \be
    \la \dot \Psi_k(t),\A_0 { \Psi_k}(t) \ra_{\Hb} 
+  i a(\Psi_k(t), \A_0 \Psi_k(t) )
+ iu(t)(b_2\Psi_k(t)-f(t),\A_0 {\Psi_k}(t) )_{\lc} = 0.
    \ee
Applying \eqref{a0yz} (in both directions) we get
  \be
\label{iazerohdkl}
    i  \la \A_0 \Psi_k(t), \A_0 \Psi_k(t) \ra_{\Hb}
 +
a( \dot \Psi_k(t) ,  \Psi_k(t) )
+ iu(t) b_2\Psi_k(t)-f(t)    ) = 0.
    \ee
Since $a(\cdot,\cdot)$ is self-adjoint we have that
\be
\ba{lll}
\ddt a(  \Psi_k(t) , \Psi_k(t) )
&=&
a(  \Psi_k(t) , \dot \Psi_k(t) )
+
a(  \dot \Psi_k(t) , \Psi_k(t) )\\
&=&
2 \Re \left(a(  \Psi_k(t) , \dot \Psi_k(t) ) \right).
\ea
\ee
So, taking real parts in \eqref{iazerohdkl}  we get
using Young's inequality and the coercivity of $a(\cdot,\cdot)$
over $\Vb$: 
    \be
    \label{testDeltaPsi2}
\ba{lll}
\half \ddt a(  \Psi_k(t) , \Psi_k(t) )&= - \Re\left(
a(  \Psi_k(t) ,  iu(t) b_2\Psi_k(t)-f(t)    ) \right)
\\ & \leq
c \| \Psi_k(t) \|_{\Vb} (\| \Psi_k(t) \|_{\Vb} + \| f(t) \|_{\Vb} )
\\ & \leq
c' ( a(  \Psi_k(t) , \Psi_k(t) ) +  \| f(t) \|_{\Vb} ).
\ea\ee
So, by Gronwall's estimate and using 
\eqref{L2_estimatek}:
\begin{equation}
\label{str_estimatek}
  \begin{aligned}
    \|\Psi_k\|_{L^{\infty}(0,T;\hc)} &\leq  c_0 \kappa[f,b_2,u,\Psi_0].
  \end{aligned}
\end{equation}

  \item 
Any $\Phi\in\hc$ can be written as 
    $ \Phi=\Phi^1+\Phi^2$ with $\Phi^1 \in \Vb_j$
and
$\Phi^2$ orthogonal to $\Vb_j$ in both spaces 
$\lc$ and $\Vb$.
 Recall the notation for the dual and antidual pairing introduced in Section \ref{sec:notation}. Then 
\be
\langle   \dot \Psi_k(t),\Phi \rangle_{\Vb} 
=\la \dot \Psi_k(t),\Phi\ra_{\Hb}
=\la \dot\Psi_k(t),{\Phi}^1\ra_{\Hb}=
\langle  \dot \Psi_k(t),\Phi^1\rangle_{\Vb}.
\ee
It follows from \eqref{Psiktest1} that
there exists $c''>0$ such that, when $\|\Phi\|_{\Vb} \leq 1,$
\be
\langle   \dot \Psi_k(t),\Phi\rangle_{\Vb} 
 \leq 
c'' \left( 
\norm{ \Psi_k(t)}{\Vb}  +
      \norm{u}{L^{\infty}(0,T)}
      \norm{b_2}{L^{\infty}(\Omega)}
      \norm{\Psi_k(t)}{\lc}      + \norm{f(t)}{\lc}
\right).
    \ee
 Combining with the above estimates we obtain
    \begin{equation}
\label{str_estimate2}
      \begin{aligned}
        \norm{\dot \Psi_k}{L^2(0,T;\Vb')}&
\leq  c_0 \kappa [f,b_2,u,\Psi_0].
      \end{aligned}
    \end{equation}
\end{enumerate}
By lemma \ref{lemwcv}
a subsequence of $(\Psi_k)$
strongly converges in 
$L^2(0,T;\Hb)$ and weakly in 
$L^2(0,T;\Vb) \cap H^1(0,T;\Vb')$,
while $u b_2 \Psi_k \rar u b_2 \Psi$
weakly in  $L^2(0,T;\bar H)$.
Passing to the limit in 
\eqref{Psiktest1}
we obtain that $\Psi$ is solution of the 
Schr\"{o}dinger equation. 
That $\Psi$ is unique, belongs to $\calx$
and satisfies \eqref{L2_estimate},\eqref{str_estimate}
and \eqref{str_estimate2}
follows from the same techniques as those used 
in the study of the Faedo-Galerkin approximation. 
\end{proof}

 \begin{lemma}
\label{lem:general_parab_setting} For $(f,b_2,u,\Psi_0)\in L^2(0,T;\Vb)\times W^{1,\infty}(\Omega)\times L^{\infty}(\Om) \times \Vb$ the semigroup solution coincides with the variational solution.
\end{lemma}

\begin{proof}

That the variational and semigroup solution coincide can be shown by a similar argument as in \cite[Lemma 5.4]{aronna:hal-01273496}. 
\end{proof}

The corresponding data of the abstract theory are
$\B_1 \in \Hb$ equal to zero, and 
 $\B_2\in \L(\Hb)$ defined by 
$(\B_2 y)(x) := -ib_2 (x) y(x)$ for $y$ in  $\Hb$ and
$x\in \Om$.
The cost function is, given $\alpha_1 \in \cR$: 
\be
\label{def-cost-fun-heat}
\ba{lll}
J(u, y) := & \disp 
\alpha_1 \int_0^T   u(t) \dd t 
+
\half\int_{(0,T)\times\Om}
( y(t,x)-y_d(t,x) )^2 \dd x \dd t 
\\ & \hspace{10mm} \disp +
 \half \int_\Om (y(T,x)-y_{dT}(x))^2 \dd x.
\ea\ee
We assume that 
\be
\label{hyp-yd-schroedinger}
\quad 
y_d \in C(0,T;\Vb); \quad  y_{dT} \in \Vb.
\ee
For $u\in L^1(0,T)$, write the reduced cost as 
$F(u) := J(u,y[u])$.
The optimal control problem is, $\U_{ad}$
being defined in \eqref{Uad}:
\be
\Min F(u); \quad u \in \U_{ad}.
\ee
\subsection{Compactness for the Schr\"odinger equation}\label{sec:control_problem}
To prove existence of an optimal control of \eqref{problem:OC} 
we have to verify the compactness hypothesis \eqref{compachyp}. 
\if{
Therefore the additional regularity in time is required we obtain by
using the equivalent weak formulation derived in Theorem
\ref{thm:reg}. \fi

\begin{proposition}
 Problem \eqref{problem:OC} for equation \eqref{Schroedinger_equ} and cost function \eqref{def-cost-fun-heat} has a nonempty set of solutions.
\end{proposition} 

\begin{proof}
This follows from theorem \ref{existsol},
whose compactness hypothesis holds thanks to
lemma \ref{lemwcv}.
\end{proof}

\subsection{Commutators}

Given $y\in\dom(\A_{\Hbc})$, we have by \eqref{A_elliptic}
that
\be
\label{rem-bracket-heat1a}
\ba{lll} 
M_1 y &=& \disp
-\sum_{j,k=1}^n \left( 
\frac{\partial b_2}{\partial x_k}  
\left[a_{jk} 
\frac{\partial y}{\partial x_j}  \right]
+ \frac{\partial }{\partial x_k}   \left[
a_{jk} 
y \frac{\partial b_2}{\partial x_j}  
\right]
\right).
\ea\ee
As expected, this commutator is
 a first order differential operator that has a continuous extension
to the space $\Vb$. 
In a similar way we can check that $[M_1,\B_2]$
is the ``zero order'' operator given by 
\be\label{m1b}
[M_1,\B_2] y =  2i \sum_{j,k=1}^n a_{j,k} 
\frac{\partial b_2}{\partial x_j}  
\frac{\partial b_2}{\partial x_k}  y.
\ee

\begin{remark}
\label{rem-bracket-heat-pc}
In the case of the Laplace operator, i.e. when
$a_{jk} = \delta_{jk}$,  we find that for $y\in \Vb$
\be
\label{rem-bracket-heat1b}
M_1 y = 
-2 \nabla b_2 \cdot \nabla y - y \Delta b_2;
\quad
[M_1,\B_2] y =  2  iy | \nabla b_2 |^2,
\ee 
and then for $p\in \Vb$ we have 
\if{
\footnote{\begin{align*}
\la M^*_1 \pb, y\ra &= \int_{\Om} (-2 \nabla b_2 \cdot \nabla y - y \Delta b_2) \bar p \dd x
=\int_{\Om} (2 \ddiv (\bar p \nabla b_2) - \bar p \Delta b_2) y \dd x
=\int_{\Om} (2 \nabla b_2 \cdot \nabla \bar p + \bar p \Delta b_2)  y \dd x
\end{align*}
}
} \fi
\be
M^*_1 p = 2 \nabla b_2 \cdot \nabla  \bar p +  \bar p \Delta b_2.
\ee
Similarly we have
\be
\left\{ \ba{lll}
M_2 y &= 2i \nabla b_2^2 \cdot \nabla y + iy \Delta b_2^2; 
\\
  \left[M_2,\B_2\right] y &=  -2  iy     | \nabla b^2_2 |^2,
\\
M^*_2 p &= -i(2 \nabla b^2_2 \cdot \nabla  \bar p +  \bar p \Delta b_2^2).
\ea\right.
\ee
\end{remark}

\subsection{Analysis of optimality conditions}
For the sake of simplicity we only discuss the case of the 
Laplace operator. 
The costate equation is then 
\be
- \dot p + i \Delta p = \Psi - \Psi_d + i u b_2 p \; 
\text{ in $(0,T)\times \Om$;}  \quad  p(T) = \Psi(T) - \Psi_{dT}. 
\ee
Remembering the expression of $b^1_{z}$
in \eqref{equ-bunz}, we obtain that 
the equation for $\xi := \xi_z$ introduced in \eqref{diffeqxi} 
reduces to 
\be
\label{xi-equ-heat}
\dot\xi - i\Delta \xi  = -i \uh b_2 \xi +w 
(ib_2 f  + 2\nabla b_2\cdot \nabla \Psi + \Psi \Delta b_2)
\; \text{ in $(0,T)\times \Om$;}  \quad  \xi(0) =0. 
\ee
The quadratic forms $\Q$ and $\hatQ$ defined in \eqref{tildeQ} and \eqref{Omega}
 are as follows. First
\be 
  \Q(z,v) = \int_0^T \Big( \norm{z(t)}{\Hb}^2
 + 2 v(t)\Re \la \ph(t),b_2  z(t) \ra_{\Hb}  \Big)\dd t + \norm{z(T)}{\Hb}^2,
 \ee
and second, 
\be
\hatQ(\xi,w,h) = \hatQ_T(\xi,h) +\hatQ_a(\xi,w)+\hatQ_b(w);
\quad \hatQ_b(w):= \int_0^T w^2(t) R(t)\dd t.
\ee
Here $R \in C(0,T)$, and 
\begin{align}
\hatQ_T(\xi,h)&:= \norm{\xi(T) -i h b_2 \Psih(T)}{\Hb}^2
 - h^2 \Re\la \ph(T),b_2^2 \Psih(T)\ra_{\Hb} +h \Re\la i\ph(T),b_2 \xi(T)\ra_{\Hb},\\
\hatQ_a(\xi,w)&:= 
 \int_0^T \Big( \norm{\xi}{\Hb}^2 + 2 w\Re(  i\la \xi,b_2 \Psih \ra_{\Hb}
 + i\la \Psih-\Psi_d,b_2\xiz\ra_{\Hb} - 
\la M^*_1 \ph, \xiz\ra_{\Hb})   
\Big) \dd t, \\
R(t)&:=  \norm{b_2\Psih}{\Hb}^2 - \Re\la\Psih-\Psi_d,b_2^2\Psih\ra_{\Hb} + \Re\la\ph(t),-b_2^2 f(t) 
- 2i|\nabla b_2|^2  \Psih\ra_{\Hb}.
\end{align}

\begin{theorem} 
{\rm (i)}
The second order necessary condition 
\eqref{lem3-ome} holds, i.e.,
\be
\label{lem3-ome-haet}
  \hatQ (\xi[w],w,h) \ge 0\quad \text{for all } (w,h) \in PC_2(\uh).
  \ee
 {\rm (ii)} 
$R(t)\geq 0$ over singular arcs. 
\\ {\rm (iii)} 
Let \eqref{finite_structures}-\eqref{RposTBB} 
hold. Then the second order optimality condition
\eqref{sufcondso} holds iff  the quadratic growth condition
\eqref{sufcondqg}  is satisfied.
\end{theorem} 

\begin{proof}
(i) Conditions \eqref{hyp-Goh-tr1a}(i) and (ii) are satisfied with \eqref{assump_Schroed}. Since we have 
\be
\overline{[-i\Delta,(-ib_2)^k]}\Psih=-(-i)^{k-1}(\Delta b_2^k \Psih +2\nabla b^k_2 \nabla \Psih), \quad k=1,2,
\ee
i.e. the commutator is a first order differential operator and has an extension to the space $\Vb$, we obtain \eqref{hyp-Goh-tr1b}(i) with $E_1=\Vb$. \eqref{hyp-Goh-tr1b}(ii) and (iii) follow from the regularity assumptions in \eqref{assump_Schroed} and \eqref{hyp-yd-schroedinger}. 

 (ii) The compactness hypothesis \eqref{comp-hyp-xi} for
\be
w\mapsto \xi[w],\quad L^2(0,T) \rightarrow L^2(0,T;\Hb) 
\ee follows from \eqref{assump_Schroed}, since hence, $\xi[w] \in
L^2(0,T;\Vb) \cap H^1(0,T;\Vb')$ which is compactly embedded in $
L^2(0,T;\Hb)$ by 
Aubin's lemma \cite{MR0152860}.

(iii) Condition \eqref{hyp-Goh-tr2} follows also from the assumptions in \eqref{assump_Schroed} and \eqref{hyp-yd-schroedinger}. 
\end{proof}


\begin{remark}
It is not difficult to extend such results
 for more general differential operators of the type,
where the $a_{jk}$ are as before, 
$b\in L^\infty(\Om)^n$ and $c\in L^\infty(\Om)$:
\be \label{general_diff_op}
(\A_{\Hbc} \Psi)(t,x)= - i \sum_{j,k=1}^n \frac{\partial}{\partial x_k} 
\left[a_{jk}(x) \frac{\partial}{\partial x_j}  \Psi (t,x) \right]
 + \sum_{j=1}^n \frac{\partial ( b_{j}(x) \Psi(t,x)) }{\partial x_j}    + c \Psi(t,x).
\ee
\end{remark}

\section{Numerical example}\label{sec:numerical_example}
The question of existence of a singular arc is not addressed here, it remains an open problem. Nevertheless, we analyze this issue numerically for the one-dimensional Schr\"odinger equation. We present a numerical example where a singular arc occurs and is stable with respect to the discretization. Let the spatial domain be given as $\Omega\subset \RR$ and set $T=10$. 
We discretize the problem by standard finite differences. In space we choose $40$ steps and in time $200$. 
For the computational realization we use the optimal control toolbox
Bocop \cite{BonMarGre12} 
which uses the nonlinear programming solver IPOPT, see 
\cite{WachBieg06}.
In Figure \ref{fig1} we see that singular arcs appear.
\begin{figure}
\centering
\includegraphics[width=0.7\textwidth]{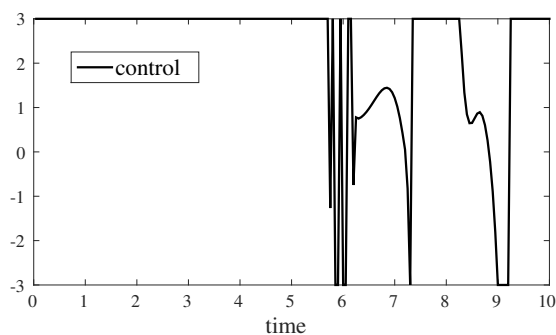}
\caption{Singular arc}\label{fig1}
\end{figure}

\if { 
\appendix

\section{Warning: not cleaned}

\subsection{Real space structure}\label{sec:absetdual-rs}
Let $\Hbc$ be a complex space, identified with 
$\H\times\H$ where $\H$ is a real space. 
We can therefore identify $\M\in \L(\Hbc)$
with a linear operator say $\M_P$ 
on $\H\times\H$, with block representation 
\be
\label{myund1a}
\M_P=
\begin{pmatrix}
      \M_{11} & -\M_{12} \\ \M_{21}  & \M_{22} 
     \end{pmatrix}.
\ee
Since $\M$ is complex linear, 
$\M (i \xb) = i \M \xb$.
It easily follows that there exists 
$\M_1$ and $\M_2$ in $\L(\H)$ such that 
\be
\label{myund1}
\M_P=
\begin{pmatrix}
      \M_1 & -\M_2\\ \M_2 & \M_1
     \end{pmatrix}.
\ee
and we can write for $y=y_1+ i y_2$ in $\Hbc$:
\be
\label{myund}
\M y = (\M_1+i\M_2)(y_1+ i y_2).
\ee
Let now $\M$ be unbounded, but with a dense domain 
in product form:
\be
\dom(\M) = \D_1\times \D_2; \quad
\D_i \subset \H, \;\; i=1,2.
\ee
Then necessarily $\D_1$ and $\D_2$ are dense 
subspaces of $\H$.
Computing $\M y$ for $y$ either in $\D_1$
or in $\D_2$, we obtain that for $y\in\dom(\M)$,
\eqref{myund} holds for some unbounded operators
$\M_1$ and $\M_2$ with domains $\D_1$ and
$\D_2$ resp., and so, 
we can still identify $\M$ with the unbounded operator in $\H\times
\H$ 
with domain $\D_1\times \D_2$, 
defined by \eqref{myund1}.

In the sequel we assume that the generator $\A$
of the semigroup has such a structure, that 
is to say:
\be
\label{myund2}
\dom(\A) = \D_1 \times \D_2.
\ee
Then as in the case of a bounded operator 
we have a block representation for $\A$
similar to \eqref{myund1a}, and since $\A$
is complex linear we deduce that 
$\A y = (\A_1+i\A_2)y$. 
So the corresponding real space operator is
\be
\begin{pmatrix}
     \A_1 & -\A_2\\ \A_2 & \A_1
     \end{pmatrix}.
\ee
It follows that
\be
\D_1 = \D_2 = \dom(\A_1)  = \dom(\A_2).
\ee

\subsection{Conjugacy}\label{sec:absetdual-rsad}
Coming back to the bounded operator $\M$,
we may compute its conjugate transpose taking either
$x=x_1$ or $x=i x_2$. It comes 
\be
\ba{lll}
\la x^*, \M y_1 \ra &= &
\la x_1+i x_2, (\M_1 +i \M_2)y_1)\ra 
\\ & =& 
\la x_1,\M_1y_1\ra
-i \la x_1,\M_2y_1\ra
+
\la x_2,\M_2y_1\ra
 +i \la x_2,\M_1y_1\ra
\\ & =&
\la \M_1^\top x_1,y_1\ra
-i \la \M_2^\top x_1,y_1\ra
+
\la \M_2^\top x_2,y_1\ra
 +i \la \M_1^\top x_2,y_1\ra
\\ & =&
\la \M_1^\top x_1 + \M_2^\top x_2
+i (\M_1^\top x_2
-\M_2^\top x_1),y_1\ra
\\ & =&
\la (\M_1^\top-i \M_2^\top) x,y_1\ra
\ea\ee
By similar computations when 
$y=i y_2$ with $y_2\in \H$, we get that for $y_1\in \H$:
\be
\la x, \M y \ra 
=
\la \M^*\xb^*, \xb \ra 
\quad
\text{where $\M^*:= \M^\top_1-i \M^\top_2$.}
\ee

Similarly, the adjoint of $\A$ has the form 
\be
\A^* = \A^\top_1 - i \A^\top_2.
\ee

{\bf} Warning: the above formula is not well justified, 
we will have to come back to it.

\section{Second order sufficient optimality conditions}

\subsection{Main results}

\subsection{See later}
We take $H$ (resp. $\bar{H}$), as the pivot space,
i.e., identified to its dual (resp. antidual),
and we let e.g. $V'$ (resp. $\Vb^*$)
denote the dual (resp. antidual) space of $V.$
Then $H$ (resp. $\bar{H}$)
 is identified to a (dense) subset of 
$V'$  (resp. $\Vb^*$), the linear 
(resp. antilinear) form
corresponding to $h\in H$ (resp. $\Hb$) being
$v\mapsto (v,h)_H$
(resp. $v\mapsto (h,v)_{\Hb}$).
We denote by $\la v^*,v\ra_V$ 
the duality product between $v\in V$ and
$v^*\in V'$. So, if $v^*\in H$ we have that
$\la v^*,v\ra_V = (v^*,v)_H$. 
We denote by $\la v^*,v\ra_{\Vb}$ 
the duality product between $v\in \Vb$ and
$v^*\in \Vb^*$.

We identify $\Hb^*$ with a subset of $V^*$
through the mapping on $V$ defined by 
$v\mapsto (h,v)_{\Hb}$. 
So when $v^*\in \Hb^*$ we have that
$\la v^*,v\ra_{\Vb} = (v^*,v)_{\Hb}$. 

} \fi

\bibliographystyle{amsplain}
\bibliography{lit,hjb}
\end{document}